\documentclass[12 pt]{article}
\usepackage{amsmath, amsthm, amssymb, bbm} 
\usepackage{enumerate}
\usepackage{amsthm, amsbsy, ifthen, graphicx}
\usepackage[active]{srcltx}

\def\cB{\mathcal{B}}

\newcommand{\scr}{\mathcal}

\newcommand{\eps}{\epsilon}

\renewcommand{\v}{\mathbf{v}}
\newcommand{\mbf}[1]{\mathbf{#1}}
\renewcommand{\P}[1]{\mathbb{P}\left[ #1 \right]}
\newcommand{\E}[1]{\mathbb{E}\left[ #1 \right]}

\newcommand{\of}[1]{\left( #1 \right) }
\newcommand{\parens}[1]{\left( #1 \right) }
\newcommand{\abs}[1]{\left| #1 \right|}

\newcommand{\sqbs}[1]{\left[ #1 \right]}
\newcommand{\braces}[1]{\left\{ #1 \right\}}

\newcommand{\qol}{z}
\newcommand{\pme}[1]{\of{1 \pm #1\eps}}

\newcommand{\rdup}[1]{{\left\lceil #1 \right\rceil }}

\theoremstyle{plain}
\newtheorem{lemma}{Lemma}
\newtheorem{defn}{Definition}
\newtheorem{fact}{Fact}
\newtheorem{thm}{Theorem}

\theoremstyle{definition}
\newtheorem{procedure}{Procedure}

\usepackage[usenames]{color}\usepackage{graphicx}
\definecolor{brown}{cmyk}{0, 0.72, 1, 0.45}
\definecolor{grey}{gray}{0.5}

\def\e{\epsilon}    
  \def\k{\kappa}
     
  \def\n{\nu} 
  \def\s{\sigma}

\newcommand{\whp}{\textbf{whp}}

\newcommand{\ignore}[1]{}

\def\cE{{\cal E}}

\def\w{{\bf w}}
\newcommand{\beq}[1]{\begin{equation}\label{#1}}
\def\eeq{\end{equation}}
\def\qs{{\bf qs}}
\def\ej{{\bf e}}
\date{}
\title{Packing tight Hamilton cycles in uniform hypergraphs}
\author{Deepak Bal and
Alan Frieze \thanks{Research supported in part by NSF award
DMS-0753472.}\\
Department of Mathematical Sciences,\\ Carnegie Mellon University,\\ Pittsburgh, PA 15213.
}
\begin{document}
\maketitle
\begin{abstract}
We say that a $k$-uniform hypergraph $C$ is a Hamilton cycle of type
$\ell$, for some $1\le \ell \le k$, if there exists a cyclic
ordering of the vertices of $C$ such that every edge consists of $k$
consecutive vertices and for every pair of consecutive edges
$E_{i-1},E_i$ in $C$ (in the natural ordering of the edges) we have
$|E_{i-1}\setminus E_i|=\ell$. We define a class of $(\e,p)$-regular hypergraphs,
that includes random hypergraphs, for which we can prove the existence
of a decomposition of almost all edges into type $\ell$ Hamilton cycles, where
$\ell<k/2$.
\end{abstract}
\section{Introduction}
This paper follows a line of work initiated by Frieze and Krivelevich \cite{FK} and
continued by Frieze, Krivelevich and Loh \cite{FKL}. We are given a $k$-regular hypergraph $H$
($k$-graph)
with certain pseudo-random properties and we show that almost almost all of the edges of
$H$ can be packed into edge disjoint Hamilton cycles of a particular type.

The paper \cite{FKL} begins with a good survey of this question which we will only give a sketch
here. When $k=2$ we are dealing with graphs. Frieze and Krivelevich \cite{FK0} showed that
the edge set of
dense graphs with a certain pseudo-random structure typified by random graphs could be almost
decomposed into edge disjoint Hamilton cycles. Knox, K\"uhn and Osthus \cite{KKO} tightened
the implied result when restricted to random graphs.

The paper \cite{FK} extended this to hypergraphs. There are various definitions of a Hamilton
cycle in a hypergraph. We will use the following:
Let $H=(V=[n],E)$ be a $k$-graph i.e. $E=\{e_1,e_2,\ldots,e_m\}$ where $e_j$
is a $k$-subset of $V$ for $j=1,2,\ldots,m$.
and let $\ell  < k$ be given where $\ell \mid n$.
A Hamilton cycle of {\em type} $\ell $ is a sequence
$f_1,f_2,\ldots,f_{\nu_\ell },\,\nu_\ell =n/\ell $ of edges
where $|g_i=f_{i+1}\setminus f_i|=\ell $ for $i=1,2,\ldots,\nu_\ell $
($f_{\nu_\ell +1}=f_1$) and $V=\bigcup_{i=1}^{\nu_\ell }g_i$. The paper \cite{FK} deals with the case
$\ell \geq k/2$ and described conditions under which almost all of the edges
of a hypergraph could be partitioned into Hamilton cycles.
The case $\ell <k/2$ could not be handled by the methods in \cite{FK}, but
\cite{FKL} shows how to deal with the case $k=3,\ell =1$. The purpose of this paper is to
extend the analysis of \cite{FKL} to the case where $k\geq 4$ and $\ell <k/2$.

We first give our notion of pseudo-randomness: We use the following notation throughout.
$$2\leq z=\rdup{\frac{k-\ell }{\ell }}\ and\ q=\ell z\qquad satisfies\ k/2<k-\ell\leq q<k.$$
\begin{defn}\label{def1}
 We say that an $n$-vertex k-graph $H$, is $(\eps,p)$-regular if the
 following holds. Let $d \in \braces{1,2,\ldots,\ell }$ and let
 $s \in \braces{1,2,\ldots,2\qol +2}$. Given any $s$ distinct
 $(k-d)$-sets, $A_1,\ldots,A_s$, such that $\abs{\bigcup_iA_i} \leq k+2q$,
 there are $(1 \pm \eps)\frac{n^d}{d!}p^s$ sets of $d$ vertices, $D$, such that all of
 $A_1\cup D,\ldots,A_s\cup D$ are edges of $H$.
 \footnote{$A=(1\pm \e)B$ if $(1-\e)B\leq A\leq (1+\e)B$}
\end{defn}

We now give our main theorem:
\footnote{The notation $a_n\gg b_n$ is short for $a_n/b_n\to\infty$ as $n\to\infty$.}
\begin{thm}\label{th1}
 Let $k$ and $\ell  < k/2$ be given. Let $\alpha = \frac{1}{9+7\qol^3}$.
 Suppose that $n$ is a sufficiently large multiple of $2q$ and that $\eps,n$ and
 $p$ satisfy  $$\eps^{16z+12}np^{8z} \gg \log^{8z+5}n.$$
 Let $H$ be an
 $(\eps,p)$-regular $k$-graph with $n$ vertices. Then $H$ contains
 a collection of edge disjoint Hamilton cycles of type $\ell $ that contains
all but at most $\eps^{\alpha}$-fraction of its edges.
\end{thm}
Our bounds on parameters $\e,p$ are unlikely to be tight and it would be interesting to sharpen our
bounds. In which case, we will not fight too hard for our bounds. In particular, we will 
replace products $(1\pm a\e)(1\pm b\e)$ and $(1\pm a\e)(1\pm b\e)^{-1}$ by $(1\pm(a+b+1)\e)$ without
further comment. Furthermore, we are really only interested in the case where $\e$ is small and so
we will always assume that $\e$ is sufficiently small for all such simplifications.

\section{Proof overview and organization}
\label{sec:overview}

The key insight in the proof of Theorem \ref{th1} is the
following connection between type $\ell $ Hamilton cycles in $H$ and Hamilton
cycles in an associated digraph.

\begin{defn}
 Given two ordered $q$-tuples of vertices $\v_1 = (v_1,\ldots,v_q),
 \v_2 = (v_{q+1},\ldots,v_{2q})$ of a $k$-uniform hypergraph $H$, we define
\beq{x1}
 \ej_i=\ej_i(\v_1,\v_2)=\braces{v_{i\ell +1}, v_{i\ell +2}, \ldots, v_{i\ell +k}}
 \quad \text{ for all } i=0,\ldots,\qol -1.
\eeq
We say that $\v_1$ precedes $\v_2$ if the edges $\ej_0,\ej_1,\ldots,\ej_{z-1}$
are all present in $H$. We say that $(\v_1,\v_2)$ {\em owns} these edges.
\end{defn}
Notice that the edges $\ej_0,\ldots,\ej_{z-1}$ are all contained in $\{v_1,v_2,\ldots,v_{2q}\}$.
$\ej_0$ consists of the first $k$ vertices of $\v_1\v_2$. We shift $\ell$ places to the right to get
$\ej_1$. We continue shifting by $\ell$ places until a further shift would take us
outside $\v_1\v_2$.

For a permutation $\s=(v_1, v_2,\ldots,v_i=\s(i),\ldots, v_n)$
of the vertices of $H$, define a $\n_q=n/q$-vertex
digraph $D_\s$ with vertex set $V_\s=\{\v_i=(v_{(i-1)q+1},\ldots, v_{iq}):\;i=1,2,\ldots,\n_q\}$.
Place an arc (directed edge) from $\v_i$ to $\v_j$ if and
only if $\v_i$ precedes $\v_j$.  In this construction, Hamilton cycles
in $D_\s$ give rise to type $\ell $ Hamilton cycles in $H$.
Indeed the Hamilton cycle $(\w_1,\w_2,\ldots,\w_{n/q})$ of $D_\s$ where
$\w_i = (w_{(i-1)q+1},\ldots,w_{iq})$ yields a Hamilton cycle in $H$ made up from the
edges owned by the arcs $(\w_i,\w_{i+1}),\,i=1,\ldots,\n_q$.
This cycle is $(e_1,e_2,\ldots,e_{\n_\ell })$ where
$e_{az+b}=\{w_{((a-1)z+b)\ell +1},\ldots,w_{((a-1)z+b)\ell +k}\}$ for $a\in[\n_q]$
and $b\in\{0,\ldots,z-1\}$.

We want disjoint Hamilton cycles in $D_\s$ to yield disjoint cycles in $H$.
This follows from the fact that the
sets of edges owned by distinct arcs $(\v_a,\v_b)$ and $(\v_c,\v_d)$ are disjoint.
Suppose then that some edge $e$ of $H$ is owned by both pairs.
It follows from the definition of precedes that
the first element of $e$ (in the order defined by $\s$) is in $\v_a$ and $\v_b$ and so $a=b$.
The $q+1$st element of $e$ is in $\v_c$ and $\v_d$ and so $c=d$, contradiction.

The basic idea of the proof is to take a large number of
random permutations $\s_1,\s_2,\ldots,\s_r$ and construct
the digraphs $D_{\s_1},D_{\s_2},\ldots,D_{\s_r}$. Then take subgraphs
$D_{\s_i}'\subseteq D_{\s_i}$ for $i=1,2,\ldots,r$ so that the edges of $H$ owned by
$D_{\s_i}' ,D_{\s_j}'$ are disjoint for $i\neq j$. It will be argued that each $D_{\s_i}'$ has
certain regularity properties implying that its arc set can be almost decomposed into edge
disjoint Hamilton cycles. We then take the edges owned by the arcs of all the Hamilton cycles
in all the $D_{\s_i}'$ and remove them to create a new hypergraph $H'$. We then argue that
\whp\ $H'$ is $(\e',p')$-regular. We repeat this process until we have covered almost all
of the edges of $H$ by Hamilton cycles.

We now give the regularity properties that we require of our digraphs:
\begin{defn}
 We say that a $\n$-vertex digraph is $(\eps,p)$-regular if it satisfies the following properties:
\begin{enumerate}[{\bf (i)}]
 \item Every vertex $a$ has out-degree $d^+(a) = (1\pm \eps)\n p$
 and in-degree $d^-(a) = (1\pm\eps)\n p.$
 \item For every pair of distinct vertices $a,b$, all three of
 the following quantities are $(1\pm\eps)\n p^2$: the number of
 common out neighbors $d^+(a,b)$, the number of common in neighbors
 $d^-(a,b)$, and the number $d^{+-}(a,b)$ of out-neighbors of $a$
 which are also in-neighbors of $b$.
 \item Given any four vertices $a,b,c,d$, which are all distinct
 except for the possibility $b=c$, there are $(1\pm\eps)\n p^4$
 vertices $x$ such that $\overrightarrow{ax}, \overrightarrow{xb},
 \overrightarrow{cx}, \overrightarrow{xd}$ are all directed edges.
\end{enumerate}
\end{defn}
In this context, we have the following Theorem of Frieze, Krivelevich and Loh \cite{FKL}:
\begin{thm}
 Suppose that $\eps^{11}np^8 \gg \log^5n$, and $n$ is a sufficiently large
 even integer. Then every $(\eps,p)$-regular digraph can have its edges partitioned
 into a disjoint union of directed Hamilton cycles, except for a set of at most
 $\eps^{1/8}$-fraction of its edges.
\end{thm}
We next describe our procedure for generating the $D_{\s_i}'$:
\begin{procedure}
 This takes as input an $(\e,p)$-regular $k$-graph $H$
 with number of vertices divisible by $q$ and an integer parameter $r$.
 Let
 \begin{equation}\label{kappa}
 \kappa = \frac{6(k+1)\log{n}}{\eps^2}\ and\ r=\frac{\ell qn^{k-2}}{ k!p^{\qol -1}}\cdot\kappa.
 \end{equation}
\begin{enumerate}[{\bf (1)}]
\item Independently generate permutations $\s_1,\s_2,\ldots,\s_r$ of $[n]$.
\item Let $H_i$ be the $k$-graph made up of the edges of $H$ that are owned by the arcs of
$D_{\s_i}$.
 \item For each edge $e\in H$, let $I_e = \braces{i : e \in H_i}$.
 If $I_e \neq \emptyset$, independently select a uniformly random
 index in $I_e$ to label $e$ with.
 \item For each $i$, define the subgraph $D_{\s_i}'$ as follows: For each arc $e=(\v,\v')$
 of $D_{\s_i}$, keep the arc $e$
if and only if all $z$ of the edges owned by $e$ are labeled with $i$.
 \item For each $i$, let $H_i'$ be the $k$-graph containing all hyperedges which are
  owned by the arcs of $D_{\s_i}'$.
\end{enumerate}
\end{procedure}

Our main task is to prove
\begin{lemma}
\label{everyDi'Uniform}
 Suppose that $n,p,$ and $\eps$ satisfy
 \[\eps^{8\qol  + 2}np^{8\qol } \gg \log^{4\qol  + 1}{n}.\]
 Let $H$ be an $(\eps,p)$-regular $k$-graph on $n$ vertices ($n$ divisible by $q$). Suppose
 that we carry out Procedure 1.  Then, with probability $1-o(n^{-1})$:
 \begin{enumerate}[{\bf (a)}]
\item Every  $D_{\s_i}'$ is $(12z^2\eps,(p/\kappa)^{\qol })$-regular.
 \item $H'$ is an $(\eps',p')$-regular $k$-graph where
$H'=H\setminus \bigcup_{i=1}^rH_i'$ is the subgraph of $H$ obtained by deleting the edges
 of the $H_i'$s. Here
 \[\eps' = \eps\of{1+\frac{7\qol^3}{\kappa^{\qol-1}}}\
 and\  p'=p\of{1-\frac{1}{\kappa^{\qol-1}}}\]
  \end{enumerate}
\end{lemma}
Part (a) enables us to find many edge disjoint Hamilton cycles and it is proved in Section
\ref{parta}. Part (b) enables us
to repeat the construction many times and is proved in Section \ref{partb}.
Section \ref{finish} shows how to use the above lemma to prove the main theorem.
\subsubsection{Random $k$-graphs}
It is as well to check that random $k$-graphs are $(\eps,p)$-regular for suitable $\e,p$.
\begin{align*}
\P{H_{n,p;k}\textrm{ is not $(\eps,p)$-regular}} &=
O(n^{k+2q})\sum_{d=1}^{\ell }\sum_{s=1}^{2\qol  +2}\P{\textrm{Bin}\sqbs{\binom{n}{d},p^s}
\neq (1 \pm \eps)\frac{n^d}{d!}p^s} \\
&= o(1)
 \end{align*}
as long as $\eps^2np^{2\qol  + 2} \gg \log n.$ (The hidden constant in $O(n^{k+2q})$ 
allows us to use $\binom{n}{d}$ in place of $\binom{n-O(1)}{d}$).

\subsection{Concentration bounds}
\begin{fact}
 For any $\eps > 0$, there exists $c_\eps > 0$ such that any
 binomial random variable $X$ with mean $\mu$ satisfies
\[\P{\abs{X-\mu} > \eps\mu} < e^{c_\eps\mu},\]
where $c_\eps$ is a constant determined by $\eps$. When $\eps < 1$,
we may take $c_\eps = \frac{\eps^2}{3}.$
\end{fact}

\begin{fact}
 Let $X$ be a random variable on the uniformly distributed
 space of permutations on $n$ elements, and let $C$ be a real
 number. Suppose that whenever $\sigma,\sigma' \in S_n$ differ
 by a single transposition, $\abs{X(\sigma) - X(\sigma')} \leq C.$ Then,
\[
\P{\abs{X - \E{X}} \geq t} \leq 2\exp\braces{-\frac{2t^2}{C^2n}}.
 \]
\end{fact}

 \subsection{Properties of $(\e,p)$-regular $k$-graphs}
\begin{lemma}
\label{proplist}
Every $n$-vertex $(\eps, p)$-regular $k$-graph $H$ has the following properties:
\begin{enumerate}[{\bf (L1)}]
 \item Given any sequence of $q$ distinct vertices, $x_1, \ldots, x_q$,
 there are $(1\pm\eps)n^{k-q}p$ sequences of vertices,
 $y_1, \ldots ,y_{k-q}$, such that $\braces{x_1,\ldots,x_q,y_1,\ldots,y_{k-q}}$
 is an edge of $H$.

 {\em In terms of Definition \ref{def1} we have $d=k-q,\,s=1,\,A_1=\{x_1,x_2,\ldots,x_q\}$.
 We multiply by $(k-q)!$ because we apply these properties to ordered sequences of vertices.}
\item Given any sequence of $k-\ell $ distinct vertices, $x_1, \ldots, x_{k-\ell }$, there
are $(1\pm\eps)n^{\ell }p$ sequences of vertices, $y_1, \ldots ,y_{\ell }$,
such that $\braces{x_1,\ldots,x_{k-\ell },y_1,\ldots,y_{\ell }}$ is an edge of $H$.

{\em In terms of Definition \ref{def1} we have $d=\ell,\, s=1,\, 
A_1 = \braces{x_1,\ldots,x_{k-\ell}}$.}
\item Given any sequence of $2q$ distinct vertices $x_1,\ldots,x_q,y_1,\ldots,y_q$,
there are $(1\pm\eps)n^{k-q}p^2$  sequences of vertices
$z_1,\ldots,z_{k-q}$ vertices such that
$\braces{x_1,\ldots,x_q,z_1,\ldots,z_{k-q}}$ and\\
$\braces{y_1,\ldots,y_q,z_1,\ldots,z_{k-q}}$ are both
edges of $H$.

{\em In terms of Definition \ref{def1} we have $d=k-q,\, s=2,\, 
A_1 = \braces{x_1,\ldots,x_{q}},\, A_2 = \braces{y_1\ldots,y_q}$.}
\item Given any sequence of $2(k-\ell )$ vertices $x_1,\ldots,x_{k-\ell },y_1,\ldots,y_{k-\ell }$
(where we demand only that $x_1 \neq y_1$), there are
$(1\pm\eps)n^{\ell }p^2$  sequences of vertices
$z_1,\ldots,z_{\ell }$ vertices such that
$\braces{x_1,\ldots,x_{k-\ell },z_1,\ldots,z_{\ell }}$ and
$\braces{y_1,\ldots,y_{k-\ell },z_1,\ldots,z_{\ell }}$ are both edges of $H$.

{\em In terms of Definition \ref{def1} we have $d=\ell,\, s=2,\, 
A_1 = \braces{x_1,\ldots,x_{k-\ell}},\, A_2 = \braces{y_1\ldots,y_{k-\ell}}$.
Note that if $\ell\mid k$,
this is identical to property (L3) since in this case, $q=k-l$.}
\item Given any sequence of $\ell +(k-2\ell )+q$ vertices
$x_1,\ldots,x_\ell , a_1,\ldots,a_{k-2\ell }, z_1,\ldots,z_q$,
there are $(1\pm\e)n^\ell p^{z+ 1}$  sequences
of vertices $b_1,\ldots,b_\ell $ such that
all of the following edges are present in $H$:
\[\braces{x_1,\ldots,x_\ell ,a_1,\ldots, a_{k-2\ell },b_1,\ldots,b_\ell }\]
and
\[
 \braces{a_{i\ell +1},\ldots,a_{k-2\ell },b_1,\ldots, b_\ell ,z_1,\ldots,z_{(i+1)\ell }}
\]
for all $i=0,\ldots, \qol -1.$

{\em In terms of Definition \ref{def1} we have $d=\ell, \, 
s=\qol+1$ and the sets $A_1,\ldots,A_{\qol+1}$ are the edges 
listed minus the set $\braces{b_1,\ldots,b_l}$.}
\item Suppose  $\ell\nmid k$.  Given any sequence of $k-\ell  + q$ distinct vertices
$a_1,\ldots, a_{k-\ell },z_1,\ldots, z_q,$
there are\\
 $(1\pm\eps)n^{q-k+\ell }p^{z}$  sequences of vertices
$b_1,\ldots, b_{q-k+\ell }$ such that all of the following edges are present in $H$:
\begin{align*}
 \braces{a_{i\ell +1},\ldots,a_{k-\ell },b_1,\ldots,b_{q-k+\ell },z_1,\ldots,z_{k-q + i\ell }},
\end{align*}
for all $i = 0,\ldots,\qol -1.$

{\em In terms of Definition \ref{def1} we have $d=q-k+\ell, \, s = \qol$, 
and the sets $A_1,\ldots, A_\qol$ are the edges listed minus the 
set $\{b_1,\ldots,b_{q-k+\ell}\}$. We require that $\ell\nmid k$ since otherwise $q-k+\ell = 0$ }

\item
Given any sequence of $2\ell + (k-2\ell ) + 2q$ distinct vertices,
\[x_1,\ldots,x_\ell , y_1,\ldots,y_\ell ,a_1,\ldots,
a_{k-2\ell }, z_1,\ldots, z_q, w_1,\ldots, w_q,\]
there are $(1\pm\eps)n^\ell p^{2\qol  + 2}$
 sequences of vertices $b_1,\ldots,b_\ell $ such
that all of the following edges are present in $H$:

\[ \braces{x_1,\ldots,x_\ell ,a_1,\ldots, a_{k-2\ell },b_1,\ldots,b_\ell },
\braces{y_1,\ldots,y_\ell ,a_1,\ldots, a_{k-2\ell }, b_1,\ldots,b_\ell }\]
and
\[\braces{a_{i\ell +1},\ldots,a_{k-2\ell },b_1,\ldots, b_\ell ,z_1,\ldots,z_{(i+1)\ell }},
\braces{a_{i\ell +1},\ldots,a_{k-2\ell },b_1,\ldots, b_\ell ,w_1,\ldots,w_{(i+1)\ell }}\]
for all $i=0,\ldots, \qol -1.$

{\em In terms of Definition \ref{def1} we have $d=\ell, \, s=2\qol+2$ 
and the sets $A_1,\ldots,A_{2\qol+2}$ are the edges listed minus the set $\{b_1,\ldots,b_l\}$.}
\item Suppose $\ell\nmid k.$ Given any sequence of $k-\ell  + 2q$ distinct vertices
$a_1,\ldots, a_{k-\ell },z_1,\ldots, z_q, w_1,\ldots, w_q,$ there
are
 $(1\pm\eps)n^{q-k+\ell }p^{2\qol }$ sequences of
vertices $b_1,\ldots, b_{q-k+\ell }$ such that all of the following edges are present in $H$:
\begin{align*}
 \braces{a_{i\ell +1},\ldots,a_{k-\ell },b_1,\ldots,b_{q-k+\ell },z_1,
 \ldots,z_{k-q + i\ell }}, \\
 \braces{a_{i\ell +1},\ldots,a_{k-\ell },b_1,\ldots,b_{q-k+\ell },w_1,\ldots,w_{k-q + i\ell }}
\end{align*}
for all $i = 0,\ldots,\qol -1.$

{\em In terms of Definition \ref{def1} we have $d = q-k+\ell, \, s=2\qol,$ 
and the sets $A_1,\ldots,A_{2\qol}$ are the sets listed minus the 
set $\{b_1,\ldots,b_{q-k+l}\}$.}
\end{enumerate}

\end{lemma}

\section{Proof of Lemma \ref{everyDi'Uniform}(a)}\label{parta}
We will follow the convention that a factor $1+o(1)$ will be absorbed into the $1\pm\e$ factors
when the $o(1)$ term is clearly small enough. This will simplify several expressions.
\begin{lemma}
\label{qTuplesInD}
 Let $S$ be a set of ordered $q$-tuples of distinct vertices with $\eps^2\abs{S}^2/n^{2q-1} \gg
 \log n$. Let $\s$ be a random permutation of $[n]$.
 Let $N=\abs{S\cap V(D_\s)}$. Then
$N = (1 \pm \eps)\frac{\abs{S}}{qn^{q-1}}$ \qs
\footnote{A sequence of events $\cE_n,n\geq 0$ is said to occur
{\em quite surely} (\qs) if $\Pr(\cE_n)=1-O(n^{-K})$ for any positive constant $K$}.
\end{lemma}
\begin{proof}
 If $\v=(v_1,\ldots,v_q)$ then
\[\Pr(\v\in V(D_\s))=\frac{1}{q}\cdot\frac{1}{n-1}\cdot\frac{1}{n-2}\cdots\frac{1}{n-q+1}
=\parens{1 \pm \frac{q^2}{2n}}\frac{1}{qn^{q-1}},\]
So
$$\E{N} = \parens{1 \pm \frac{q^2}{2n}}\frac{|S|}{qn^{q-1}}.$$

Suppose the permutation $\sigma$ is converted
to $\sigma'$ by a single transposition. Then this
changes at most 2 of the vertices of $D_\s$. So $N$ can change by at most 2.
Then Fact 2 implies that the probability that $N$ deviates from its mean by more
than $\frac{\eps}{2}\frac{\abs{S}}{qn^{q-1}}$  is at most
\[
 2\exp\braces{-\frac{2\of{\frac{\eps}{2}\frac{\abs{S}}{qn^{q-1}}}^2}{2^2n}} = O(n^{-K})
\]
for any positive constant $K$. The lemma follows since $q^2/n=O(1/n)\ll \e$.

\end{proof}
\begin{lemma}
\label{everyDiUniform}
 Suppose $n,p,$ and $\eps$ satisfy $\eps^2np^{8\qol } \gg \log n.$ Let $H$ be an
 $(\eps,p)$-regular $k$-graph on $n$ vertices ($n$ divisible by $q$).
 Let $\s$ be a random permutation of $[n]$.
Then $D=D_\s$ is $((2\qol +5)\eps,p^{\qol })$-regular, \qs.
\end{lemma}

\begin{proof}
 We verify the properties of $D$ one at a time, starting with out-degrees.
 Fix any $q$ vertices, $v_1,\ldots,v_q$. Let $\v = (v_1,\ldots,v_q)$. Let
 $N_\v$ be the number of $q$-tuples $\mathbf{w}$
 such that (a) $\w\in V(D)$ and (b) $\v$ precedes $\mathbf{w}$. It suffices to show that
 with probability $1-o(n^{-(q+1)})$, $N_\v = \pme{(2\qol  + 5)}p^z\n_q$.
 Let $S_\v$ be the set of $q$-tuples $\mathbf{w}$,
 such that $\v$ precedes $\mathbf{w}$.

Apply property (L1) of Lemma \ref{proplist} to $\braces{v_1,\ldots v_q}$ and fix
one of the $\pme{}n^{k-q}p$ sequences $(v_{q+1}, \ldots ,v_{k})$ such that $\{v_1,\ldots,v_k\}\in H$.
Let $\mathbf{u} = (v_1,\ldots,v_k)$
and do the following $\qol -1$ times:
\begin{enumerate}
 \item Apply property (L2) of Lemma \ref{proplist} to the trailing $k-\ell $
 elements of $\mathbf{u}$.
 \item Fix one of the $\pme{}n^\ell p$ sequences of $\ell $ vertices.
 \item Append this sequence of $\ell $ vertices to the end of $\mathbf{u}$.
\end{enumerate}

At the end of this process, $k-q+(z-1)\ell =k-\ell $
distinct vertices, $(v_{q+1},\ldots, v_{q+k-\ell })$,
have been fixed and appear at the trailing end of $\mathbf{u}$.
Fix any $q-k+\ell\geq 0$ distinct vertices to give the $q$ tuple
$\mathbf{w} = (v_{q+1},\ldots,v_{2q})$.

Combining our estimates from each step tells us that
\begin{align*}
 \abs{S_\v}&=\pme{}n^{k-q}p \cdot\of{\pme{}n^\ell p}^{\qol -1}\cdot n^{q-k+\ell}\\
 &=\pme{(2\qol +3)}n^qp^{\qol }\\
\noalign{ and so}
\E{N_\v}&=\frac{\E{|S_\v|}}{q(n-1)\cdots(n-q+1)}=\frac{\pme{(2\qol +4)}n^qp^{\qol }}{qn^{q-1}}
=\pme{(2\qol +4)}p^z\n_q.
\end{align*}
Since $\eps^2p^{\qol}n \gg \log n,$ we can apply Lemma
\ref{qTuplesInD} to $S_\v$ to conclude that \qs
\[
 N_\v = \pme{(2\qol  + 5)}p^{\qol }\n_q.
\]
For in-degrees, fix a $q$-tuple \[\mathbf{u} = \v = (v_{q+1},\ldots,v_{2q}).\]
do the following $\qol $ times:
\begin{enumerate}
 \item Apply property (L2) of Lemma \ref{proplist}
 to the leading $k-\ell $ elements of $\mathbf{u}$.
 \item Fix one of the $\pme{}n^\ell p$ sequences of $\ell $ vertices.
 \item Prepend this sequence to the beginning of $\mathbf{u}$.
\end{enumerate}
At the end of this process, $q$ vertices have been fixed and appear in the
first $q$ positions of $\mathbf{u}$. Call this $q$-tuple $\mathbf{w}$.
Combining estimates from each step of the process tells us that the number
of such $\mathbf{w}$ that precede $\v$ is
\[
 \of{\pme{}n^\ell p}^{\qol } = \pme{(2\qol +1)}n^qp^{\qol }.
\]
Applying Lemma \ref{qTuplesInD} as before gives us that \qs\ the in-degree of $\v$ in $D$ is
\[\pme{(2\qol +3)}p^{\qol }\n_q.
\]
The remaining properties are dealt with in a similar manner. For each, we will state
what properties from Lemma \ref{proplist} to apply and compute the number of
satisfying $q$-tuples. In all cases, an application of Lemma \ref{qTuplesInD} completes the argument.

For $d^+(\mbf{x},\mbf{y})$ in $D$, fix 2 $q$-tuples of distinct
vertices, $\mbf{x} = (x_1,\ldots,x_q)$
and $\mbf{y} = (y_1,\ldots,y_q)$ and apply property (L3)
to obtain $(z_1,z_2,\ldots,z_{k-q})$ in
$(1\pm\e)n^{k-q}p^2$ ways. Follow by $\qol -1$ applications
of property (L4).
Our first iteration applies (L4) to $x_{\ell+1},\ldots,x_q,z_1,\ldots,z_{k-q}$
and $y_{\ell+1},\ldots,y_q,z_1,\ldots,z_{k-q}$ to obtain $(z_{k-q+1},\ldots,z_{k-q+\ell})$
in $(1\pm\e)n^{k-q}p^2$ ways. We then shift right $\ell$ terms
along both sequences and apply (L4) again. In our last application we feed sequences that begin
with $x_{(z-1)\ell+1}\neq y_{(z-1)\ell+1}$ using the fact that $(z-1)\ell+1<q+1$.
Arbitrarily choose
$q-k+\ell\geq0$ more vertices to fill out $z_1,\ldots,z_q$. The estimate in this case is
\[
 (1\pm\eps)n^{k-q}p^2\cdot\of{(1\pm\eps)n^{\ell }p^2}^{\qol -1}\cdot(n-(k-\ell ))\cdots(n-(q-1))
\]
Simplifying and applying Lemma \ref{qTuplesInD} gives that $d^+(\mbf{x},\mbf{y})$ in $D$ is \qs
\[
 \pme{(2\qol  + 5)}p^{2\qol }\n_q.
\]
Similarly $d^-(\mbf{x},\mbf{y})$ is \qs
\[\pme{(2\qol  + 5)}p^{2\qol }\n_q.\]
For $d^{+-}(\mbf{x},\mbf{y})$ in $D$, fix $2q$ distinct vertices
arranged in 2 $q$-tuples, $\mbf{x} = (x_1,\ldots,x_q)$ and $\mbf{y} = (y_1,\ldots,y_q).$
If $\ell $ divides $k$, (so that $q=k-\ell$),
apply property (L2) $\qol -1$ times starting with $\mbf{x}$. After the 
first iteration, we obtain  $(z_1,\ldots,z_\ell)$ in
$\pme{}n^{l}p$ ways. We shift right by $\ell$
in the sequence for each subsequent application of property (L2) to 
obtain $(z_1,\ldots,z_{q-\ell})$. Note here that $q-\ell = k-2\ell$.
Property (L5) is then applied to $x_{q-\ell+1},\ldots,x_{q},z_1,\ldots,z_{q-\ell},y_1,\ldots,y_q$.
The estimate in this case is
\[
 \of{\pme{}n^\ell p}^{z-1}\cdot \pme{}n^\ell p^{z+1}
\]
If $\ell $ does not divide $k$, then apply (L1) to $\mbf{x}$ to 
obtain $(z_1,\ldots,z_{k-q})$ in $\pme{}n^{k-q}p$ ways.
Follow this by $\qol -1$ applications of
(L2), shifting right by $\ell$ in the sequence for each application to 
obtain $(z_1,\ldots,z_{k-\ell})$.  Follow by an
application of (L6) to $z_1,\ldots,z_{k-\ell},y_1,\ldots,y_q$ to fill 
out $(z_1,\ldots,z_q)$. The estimate in this case is
\[
 \pme{}n^{k-q}p\cdot\of{\pme{}n^\ell p}^{z-1}\cdot\pme{}n^{q-k+\ell}p^z
\]
Simplifying and applying Lemma \ref{qTuplesInD} in both cases gives 
that $d^{+-}(\mbf{x},\mbf{y})$ is \qs
\[
 \pme{(2\qol +5)}p^{2\qol }\n_q
\]

For the third property of digraph uniformity, fix $4q$ distinct vertices arranged
in 4 $q$-tuples, $\mbf{x} = (x_1,\ldots,x_q), \,\mbf{y} = (y_1,\ldots,y_q), \, 
\mbf{z} = (z_1,\ldots,z_q),$ and $\mbf{w} = (w_1,\ldots,w_q)$. If $\ell $ divides 
$k$, do $\qol -1$ applications
of property (L4).
Our first iteration applies (L4) to $x_{1},\ldots,x_q,y_1,\ldots,y_q$
 to obtain $(a_{1},\ldots,a_{\ell})$
in $(1\pm\e)n^{\ell}p^2$ ways. We then shift right $\ell$ terms
along both sequences and apply (L4) to  $x_{\ell+1},\ldots,x_q,a_1,\ldots,a_{\ell}$
and $y_{\ell+1},\ldots,y_q,a_1,\ldots,a_{\ell}$ and so on until we have 
obtained $(a_1,\ldots,a_{k-2\ell})$. We then apply property (L7) to
\[
 x_{q-\ell+1},\ldots,x_q,y_{q-\ell+1},\ldots,y_q,a_1,\ldots,a_{k-2\ell},z_1,\ldots,z_q,w_1,\ldots,w_q
\]
to find $(a_{k-2\ell +1}, \ldots, a_q)$. The estimate in this case is
\[
  \of{\pme{}n^\ell p^2}^{\qol-1}\cdot \pme{}n^\ell p^{2\qol+2}.
\]
If $\ell $ does not divide $k$, apply property (L3) to $x_1,\ldots,x_q,y_1,\ldots,y_q$ 
to obtain $(a_1,\ldots,a_{k-q})$ in $\pme{}n^{k-q}p^2$.
Follow by $\qol -1$ applications of (L4) as in the proof of $d^+(\mbf{x},\mbf{y})$ 
to obtain $(a_1,a_2,\ldots,a_{k-\ell})$.
Then apply (L8) to $a_1,\ldots,a_{k-\ell},z_1,\dots,z_q,w_1,\ldots,w_q$ in order 
to find $(a_{k-\ell+1},\ldots,a_q)$. The estimate in this case is
\[
 \pme{}n^{k-q}p^2\cdot\of{\pme{}n^\ell p^2}^{\qol-1}\cdot\pme{}n^{q-k+\ell}p^{2\qol}.
\]
Simplifying and applying Lemma \ref{qTuplesInD} in both cases gives \qs
\[
   \pme{(2\qol +5)}p^{4\qol}\n_q
\]
for property (iii) of digraph uniformity.
\end{proof}

\begin{lemma}
\label{everyEdgeCovered}
 Suppose $n,p$, and $\eps$ satisfy $\eps n \gg 1$. Let $H$ be an $(\eps,p)$-regular
 $k$-graph on $n$ vertices ($n$ divisible by $q$), and randomly and independently
 construct digraphs $D_1,\ldots,D_r$ according to Procedure 1. Let $H_i$ be their corresponding
$k$-graphs. Then with probability $1-o(n^{-1})$, every edge of $H$ is an edge in
$\pme{(\qol +2)}\kappa$ of the $H_i$. Here $\kappa,r$ are as defined in \eqref{kappa}.
\end{lemma}

\begin{proof}
We must first calculate the probability that an edge of $H$ appears in an $H_i$
after Procedure 1. This probability is
\[
 p_1= \frac{k!(1\pm z\e)p^{z-1}}{\ell qn^{k-2}}.
\]

To see this, first fix an edge $e=\braces{x_1,\ldots,x_k}$ of $H$. We want the
probability that this is an edge of $H_1$, say.  For this to happen, there must
be two vertices $\v_1=(v_1,\ldots,v_q),\v_2=(v_{q+1},\ldots,v_{2q})$ of $D_1$ and
$0\leq i\leq z-1$
such that $e=e_i(\v_1,\v_2)$. Fix such an $i$. We now have to consider the number of choices for
$v_1,\ldots,v_{i\ell},v_{i\ell+k+1},\ldots,v_{(z-1)\ell+k+1},\ldots,v_{2q}$.
The $(\e,p)$-regularity of $H$ implies that there will be
$$((1\pm\e)n^\ell p)^{z-1}n^{2q-(z-1)\ell-k}=(1\pm (z-.5)\e)p^{z-1}n^{2q-k}$$
choices for this
sequence.

The probability that $\v_1,\v_2$ are vertices of $H_1$ is
\[p_2=\parens{\frac{1}{q}\cdot\of{\prod_{i=1}^{q-1}{\frac{1}{n-i}}}}^2.\]
Now there are $z$ choices for $i$ and $k!$ choices for the ordering of $e$ and so the
probability that $e$ is an edge of $H_1$ is
$$zk!(1\pm (z-.5)\e)p^{z-1}n^{2q-k}p_2=p_1.$$

Since the $r$ random constructions are independent, the number $Z_e$ of $H_i$ that contain
$e$ is distributed as $Bin[r,p_1]$. So,
\[\E{Z_e} = rp_1 = \pme{z}\kappa.\]
So the Chernoff bound tells us the probability that this Binomial deviates from
its mean by more than a factor of $1\pm\eps$ is at most
\[
 2\exp\braces{-\frac{\eps^2}{3}\cdot\of{1-z\eps}\kappa} = o(n^{-k-1}).
\]
So taking a union bound over all $O(n^k)$ choices for $e$ gives the result.
\end{proof}

\noindent
{\bf Proof of Lemma \ref{everyDi'Uniform}(a):}
Our conditions on $n,p$ and $\eps$ allow us to apply Lemmas
 \ref{everyDiUniform}  and \ref{everyEdgeCovered}. So with probability
 $1-o(n^{-1})$, after Step 1 of Procedure 1,
\begin{itemize}
 \item[(a)] Every $D_i$ is $((2\qol +5)\eps,p^{\qol })$-regular.
 \item[(b)] Every edge in $H$ is covered $\pme{(\qol +2)}\kappa$ times by the $H_i$.
\end{itemize}
Condition on the above outcome of Steps 1 and 2, and consider an arbitrary $D_i'$
(as defined in Step 4 of Procedure 1.
$r=o(n^{k-1})$ (since $\eps^2np^{\qol -1} \gg \log n$), so it suffices to show
that with probability $1-o(n^{-k})$, $D_1'$ has the desired properties.

For out-degrees: A vertex $\v\in D_1'$ corresponds to a $q$-tuple of vertices in $H$.
An edge $e$ of $D_1$ remains in $D_1'$ if and only if all the $\qol $ hyperedges of $H$
owned by $e$
receive label 1 in Step 3.  This happens with probability
\[
\frac{1}{\sqbs{\pme{(\qol  + 2)}\kappa}^{\qol }} = \pme{(\qol^2  + 2\qol  + 1)}
\frac{1}{\kappa^{\qol }}
\]

There are $\pme{(2\qol  + 5)}\n_qp^{\qol }$ neighbors of $\v$ in $D_1$,
so the expected out-degree of $\v$ in $D_1'$ is
\[
 \pme{(2\qol  + 5)}\pme{(\qol^2  + 2\qol  + 1)}\n_q \of{\frac{p}{\kappa}}^{\qol }
 = \pme{(\qol^2  + 4\qol  + 7)}\frac{n}{q} \of{\frac{p}{\kappa}}^{\qol }.
\]

For concentration, the Chernoff inequality tells us that the probability that the out-degree
of vertex $\v$ in $D_1'$ deviates from its expectation by more than a factor
of $1\pm \eps$ is at most
\[
 2\exp\braces{\frac{\eps^2}{3}\cdot\of{1-(\qol^2  + 4\qol
 + 7)\eps}\frac{n}{q} \of{\frac{p}{\kappa}}^{\qol } } \leq o(n^{-k-1})
\]
as long as
\[
 \frac{\eps^2np^{\qol }}{\kappa^{\qol }} =
 \Theta\of{\frac{\eps^{2\qol  + 2}np^{\qol }}{\log^{\qol }n}} \gg \log n.
\]
This is true by our assumptions on $n,p$ and $\eps$.  Therefore with probability
$1-o(n^{-k-1})$, the out degree of $\v$ in $D_i'$ is $\pme{(\qol^2  + 4\qol  +
 9)}\n_q \of{\frac{p}{\kappa}}^{\qol }$. Taking a union bound over all $O(n)$
 vertices in $D_1'$ establishes uniformity for out-degrees.

The other properties follow from a similar argument. The smallest mean we deal
with is in property (iii) of digraph regularity:
\[
 \frac{n(1\pm(2z+5)\e)p^{4z}}{q((1\pm(z+2)\e)\k)^{4z}}=\pme{(4\qol^2  + 10\qol  + 7)}
 \frac{n}{q} \of{\frac{p}{\kappa}}^{4\qol }.
\]
So the error in concentration is at most
\[
  2\exp\braces{\frac{\eps^2}{3}\cdot\of{1-(4\qol^2  + 10\qol
  + 7)\eps}\frac{n}{q} \of{\frac{p}{\kappa}}^{4\qol } } \leq o(n^{-k-4})
\]
as long as  $\eps^{8\qol  + 2}np^{4\qol }/\log^{4\qol }n \gg \log n$ which it
is by assumption. Taking a union bound over all $O(n^4)$ choices for vertices
in $D_1'$ gives the result.
\qed
\section{Proof of Lemma \ref{everyDi'Uniform}(b)}\label{partb}
We will be applying the principle of inclusion-exclusion to get an estimate on the
regularity of $H'$. So we use the next two Lemmas to compute a first
order estimate and a second order upper bound on several quantities.

Given a hyperedge $e$ and a digraph $D_i$ from Procedure 1, edge $e$ is owned
by at most one directed edge in $D_i$. If this edge exists, let it be denoted $u_i(e)$.
Now $u_i(e)$ owns exactly $\qol $ hyperedges in $H_i$.  If $e$ is
is an edge of $H_i$, let $\phi_i(e)$ be the set of $\qol$ hyperedges owned by $u_i(e)$.
Note that $\phi_i(e)$ includes the edge $e$. We call $\phi_i(e) \backslash\braces{e}$
the \emph{partner} edges of $e$ in $H_i$.

\begin{lemma}
\label{firstOrder}
Condition on $|I_e|=\pme{(\qol  + 2)}\kappa$ for each edge of $H$.
Fix $d \in \braces{1,\ldots,\ell }$ and any set of $k-d$ vertices, $A =
 \braces{a_1,\ldots,a_{k-d}} \subset V(H)$.  Fix a family $\mathcal{B}$  of $d$-sets
 of vertices such that $A\cup B$ is a hyperedge of $H$ for all $B\in \mathcal{B}$.
 Suppose $\eps^2\abs{\mathcal{B}}/\kappa^{2\qol -1} \gg \log n$. Then
 with probability $1-o(n^{-(k+2q) -1})$, the number $N_\cB$ of $B\in \mathcal{B}$ such
 that $A\cup B\in \bigcup_iE(H_i')$ satifies
 $N_\cB=\pme{(\qol^2  + \qol)}\frac{\abs{\mathcal{B}}}{\kappa^{\qol  -1}}$
\end{lemma}
\begin{proof}
 Let $\mathcal{B} = \braces{B_1,\ldots,B_t}.$ Because we are conditioning on $|I_e|,e\in E(H)$,
 the relevant probability space is the choice of labels in Step 3 of Procedure 1.
 Define $F=F(A)$, the set of relevant edges, as
 follows: For each $j$ such that $A\cup B_i\in E(H_j)$ there are exactly $\qol -1$
 partner edges $F_{i,j}$ such that $A\cup B_i\in E(H_j')$ if and only if all of these edges as
 well as $A\cup B_i$ receive label $j$.  Let $F=\bigcup_{i,j}F_{i,j}$.
 Since we assume that each edge is in
 $\pme{(\qol  + 2)}\kappa$ of the $H_j$, we have that
 $\abs{F} \leq 2\qol \kappa\abs{\mathcal{B}}$. The labels outside of $F$ do not affect
 the count $N$, so we may condition on an arbitrary setting of those labels leaving only
 the labels of $F$ to be exposed.

Now
\[
	\Pr\left[A\cup B_i\in\bigcup_j E(H_j')\right]=\sqbs{\pme{(\qol  + 2)}\kappa}^{-(\qol -1)}.
\]
To see this, expose the label of an edge $A\cup B_i$. Suppose that it receives label $j$.
Then all of its partner edges must also receive label $j$. Each of them is
an edge of $\pme{(\qol  + 2)}\kappa$ of the $H_k$, and since their labelings are
independent, the probability that each of them receive label $j$ is as claimed above.
So
\[
 \E{N_\cB} = \abs{\mathcal{B}}\sqbs{\pme{(\qol  + 2)}\kappa}^{-(\qol -1)} =
 \pme{(\qol^2  + \qol  -1)}\frac{\abs{\mathcal{B}}}{\kappa^{\qol -1}}
\]

Our probability space is a product space of dimension $\abs{F}$.
We use the Hoeffding-Azuma inequality to show that $N_\cB$ is concentrated.
Suppose the label of an edge $e\in F$ is changed from $i$ to $j$. Suppose that
$e$ is owned by the edge $(\v_1=(v_1,\ldots,v_q),\v_2=(v_{q+1},\ldots,v_{2q}))$ of $D_i$.
Let $S=\{v_1,\ldots,v_{2q}\}$.
The definition of $F$ implies that $S\supseteq A$.
So at most $\binom{2q-(k-d)}{d}$ sets from $\mathcal{B}$ will be removed from the count
$N_\cB$ by this switch in labels. Similarly, at most $\binom{2q-(k-d)}{d}$ sets from $\mathcal{B}$
will be added to the count $N_\cB$. Hence $N_\cB$ is $\binom{2q-(k-d)}{d}$-Lipschitz and the
Hoeffding-Azuma inequality implies that the probability that $N_\cB$ deviates from its
mean by more than $\eps\abs{\mathcal{B}}/\kappa^{\qol  -1}$ is at most
\[
 2\exp\braces{-\frac{\of{\eps\abs{\mathcal{B}}/\kappa^{\qol  -1}}^2}{2\binom{2q-(k-d)}{d}^2\abs{F}}}
\leq2\exp\braces{-\frac{\eps^2\abs{\mathcal{B}}}{4\qol \binom{2q-(k-d)}{d}^2\kappa^{2\qol  -1}}}
\leq o(n^{-(k+2q) -1})
\]
as long as $\eps^2\abs{\mathcal{B}}/\kappa^{2\qol  -1} \gg \log n$, which we assumed.
Therefore $N=\pme{(\qol^2  + \qol)}\frac{\abs{\mathcal{B}}}{\kappa^{\qol -1}}$
with the desired probability.
\end{proof}

Let $1\leq t\leq 2q-k$. Let $D_i$ be a digraph constructed from Procedure 1.
Say that a set $S$ of $k+t$ vertices is \emph{condensed} in $D_i$ if
there exist edges $e_1\neq e_2$ of $H$ such that $S=e_1\cup e_2$ and $\phi_i(e_1)\cap \phi_i(e_2)
\neq\emptyset$.

\begin{lemma}
\label{condensed}
 Suppose $r\ll n^{k- \frac{3}{2}}$  Construct $r$ independent $D_i$ according to
 Procedure 1.  Then with probability $1-o(n^{-1})$, every set of $S$ of $k+t$ vertices,
 $1\leq t \leq 2q-k$, is condensed in at most $4q+1$ of the $D_i$.
\end{lemma}
\begin{proof}
 Fix a set of $k+t$ vertices $S=\braces{x_1,x_2,\ldots,x_{k+t}}=e_1\cup e_2$
 where $e_1,e_2$ are edges of $H$.
The probability that $S$ is condensed in $D_1$ is at most
\[
 (k+t)!\cdot\frac{1}{q}\cdot(\qol-1)
 \cdot\of{\prod_{i=1}^{q+t-1}\frac{1}{n-i}}
 < \frac{(2q)!}{\ell n^{k+t-1}}
\]
This calculation is very similar to the one in Lemma \ref{everyEdgeCovered}.

Since the $D_i$ are independent, the number of them which have the above
property with respect to $S$ is stochastically dominated by
 $\textrm{Bin}\sqbs{r,\frac{(2q)!}{\ell n^{k+t-1}}}$.  Since we
 assumed that $r \ll n^{k-\frac{3}{2}}$, the probability that this exceeds $4q+1$ is at most
\[
 \binom{r}{4q+2}\of{\frac{(2q)!}{\ell n^{k+t-1}}}^{4q+2} =
 o(n^{(k-\frac{3}{2} -k - t + 1)(4q+2)}) = o(n^{-2q-1})
\]

Now taking a union bound over all $O(n^{2q})$ choices for $S$ gives the result.
\end{proof}

 \begin{lemma}
 \label{secondOrder}
Condition on $|I_e|=\pme{(\qol  + 2)}\kappa$ for each edge of $H$. Also condition on the property
  that every set of $k+t$ vertices, $1\leq t\leq 2q-k$, is condensed in at most
  $4q+1$ of the $D_i$. Fix $d\in \braces{1,\ldots \ell }$ and any 2 sets, $A_1$ and
  $A_2$ of $k-d$ vertices. Fix a family $\scr{B}$ of $d$-sets of vertices such
  that $A_1\cup B$ and $A_2 \cup B$ are both hyperedges of $H$ for all $B\in \scr{B}.$
  Suppose $\abs{\scr{B}}/\kappa^{2\qol +1} \gg \log n$. Then with probability
  $o(n^{-(k+2q)-1})$, the number $N_\cB$ of $B\in\scr{B}$ such that
  $A_1\cup B\in\bigcup_iH_i'$ and $A_2\cup B\in\bigcup_iH_i'$ is
  at most $7q\abs{\scr{B}}/\kappa^{\qol }$
 \end{lemma}

\begin{proof}
Let $\mathcal{B} = \braces{B_1,\ldots,B_t}$ and
let $F^*=F(A_1)\cup F(A_2)$ where $F$ is as defined in Lemma \ref{firstOrder}.
 Then $\abs{F^*} \leq 3\qol \kappa\abs{\scr{B}}$.  We would like an upper bound on the
 probability that a particular $B\in\scr{B}$ contributes to $N_\cB$. Let
 $e_1 = A_1\cup B$ and $e_2 =A_2\cup B$. First, expose the label of  $e_1$
 and suppose it is $j$.

\emph{Case 1:} $e_2$ receives label $j$.

 If $\phi_k(e_1)\cap\phi_k(e_2)=\emptyset$, then the probability that $e_1,e_2\in H_j'$ is at most
\[
 q_1 := \sqbs{\of{1-{(\qol  + 2)}\eps}\kappa}^{-(2\qol  -1)}.
\]
To see this, note that the probability that $e_2$ receives label $j$  is
$\of{\pme{(\qol  + 2)}\kappa}^{-1}$, and since their $2(\qol -1)$ partner edges
are distinct and labelings are independent, we get the desired probability.

If $\phi_j(e_1)\cap\phi_j(e_2)\neq\emptyset$ then
the vertices of $e_1\cup e_2$ are condensed in $D_j$. We have
$k+1 \leq \abs{e_1\cup e_2}\leq 2q$, so by assumption, these vertices are condensed in
at most $4q+1$ of the $D_i$.  So the probability that $e_1$ and $e_2$ are both
in $E(H_j')$ is bounded above by
\[
 q_2 =\frac{4q+1}{\of{1-{(\qol  + 2)\eps}}\kappa}\cdot
 \frac{1}{ \sqbs{\of{1-\of{\qol  + 2}\eps}\kappa}^{\qol  -1}} =
 \frac{4q+1}{ \sqbs{\of{1-\of{\qol  + 2}\eps}\kappa}^{\qol }}
\]
since all of the partner edges of $e_1$ must also receive label $j$.

\emph{Case 2:} $e_2$ receives label $l \neq j$.

If $\phi_j(e_1)\cap\phi_l(e_2) = \emptyset$ then the probability that everything
receives the appropriate label is at most
\[
  q_3=\sqbs{\of{1-{(\qol  + 2)\eps}}\kappa}^{-(2\qol  -2)}.
\]

If $\phi_j(e_1)\cap\phi_l(e_2) \neq \emptyset$, then the probability that $B$ contributes
to $N_\cB$ is 0 since an edge in the intersection must receive both labels $j$ and $l$.

Summing up these upper bounds, we get that the probability that $B$ contributes to $N_\cB$
is bounded above by
\[
 q_1+q_2+q_3 \leq \frac{4q+3}{ \sqbs{\of{1-\of{\qol  + 2}\eps}\kappa}^{\qol }}
 \leq \frac{6q}{\kappa^{\qol }}.
\]
So $\E{N_\cB}\leq \frac{6q}{\kappa^{\qol }}\abs{\scr{B}}$.
By a similar argument as in Lemma \ref{firstOrder}, we can see that $N_\cB$ is
$\binom{2q-(k-d)}{d}$-Lipschitz in the product space of dimension
$\abs{F^*}\leq 3\qol \kappa\abs{\scr{B}}$. So the probability that $N_\cB$ exceeds
its expectation by more than $\abs{\scr{B}}/\kappa^{\qol }$ is at most
\[
 2\exp\braces{-\frac{\of{\abs{\scr{B}}/\kappa^{\qol }}^2}{2\cdot\binom{2q-(k-d)}{d}^2\abs{F^*}}}
\leq 2\exp\braces{-\frac{\abs{\scr{B}}}{6\qol \cdot\binom{2q-(k-d)}{d}^2\cdot\kappa^{2\qol +1}}}
\leq o(n^{-(k+2q) -1})
\]
since we assumed that $\abs{\scr{B}}/\kappa^{2\qol +1}\gg \log n$. Therefore
$N_\cB\leq \frac{7q}{\kappa^{\qol }}\abs{\scr{B}}$ with the desired probability.
\end{proof}

\noindent
{\bf Proof of Lemma \ref{everyDi'Uniform}(b):}
By applying Lemma \ref{everyEdgeCovered} and Lemma \ref{condensed}, the conditions
of which hold by our requirements on $n,p$ and $\eps$, the outcome of Steps 1 and 2 of
Procedure 1 satisfies the following with probability $1-o(n^{-1})$.
\begin{itemize}
 \item Every edge of $H$ is covered $\pme{(\qol + 2)}$ by the $H_i$.
 \item Every set of $k+t$, $1\leq t\leq 2q-k$ vertices is condensed in at most
 $4q+1$ of the $D_i$.
\end{itemize}
Condition on this outcome. We will show that in the context of the choices in Step 3,
$(\eps',p')$-regularity is satisfied with probability $1-o(n^{-1})$.

Fix $d\in \braces{1,\ldots,\ell }$, $s\in \braces{1,\ldots,2\qol +2}$ and a family of
$s$ distinct $(k-d)$-sets $\Gamma = \braces{A_1,\ldots,A_s}$ with $\abs{\cup_i A_i}
\leq k+2q$. Let $X$ be the number of $d$-sets, $B$, such that $A_i \cup B$ is an
edge of $H'$ for all $i=1,\ldots,s$. It suffices to show that
$X=(1\pm \eps')\frac{n^d}{d!}p'^s$ with probability $1-o(n^{-(k+2q)-1})$.
Then we can use the union bound over all $O(n^{k+2q})$ choices for vertices $\abs{\cup_i A_i}$
and all $O(1)$ choices of set families on those vertices.

 Let $\scr{B}$ be the family of all $d$-sets $B$ such that $A_i \cup B$ are edges
 of $H$ for all $i=1,\ldots, s$ and $B\in \scr{B}$. $H$ is $(\eps,p)$-regular, so
 $\abs{\scr{B}} = \pme{}\frac{n^d}{d!}p^s$.

For each $i\in\{1,\ldots,s\}$, let $X_i$ be the number of elements $B$ of $\scr{B}$
with $A_i\cup B\in\bigcup_lH_l'$. For every $i,j\in \braces{1,\ldots,s}, i\neq j$,
let $X_{ij}$ be the number of elements, $B$, of $\scr{B}$ with both $A_i\cup B\in\bigcup_lH_l'$ and
$A_j \cup B\in\bigcup_lH_l'$.

Then
\[
 \abs{\scr{B}}-\sum_{i=1}^s X_i \leq X \leq \abs{\scr{B}} - \sum_{i=1}^s X_i + \sum_{i<j}X_{ij}.
\]
Note that since $d\geq 1$ and $s\leq 2\qol +2$, we have
\[\abs{\scr{B}} = \Theta\of{n^dp^s} = \Omega\of{np^{2\qol+2}}.\]
We apply Lemmas \ref{firstOrder} and \ref{secondOrder}.
Indeed, by our requirements on $n,p$ and $\eps$ we have both
\[
 \frac{\eps^2\abs{\scr{B}}}{\kappa^{2\qol-1}}
=\Omega\of{\frac{\eps^{4\qol}np^{2\qol+2}}{\log^{2\qol-1}n}} \gg \log n
\]
and
\[
 \frac{\abs{\scr{B}}}{\kappa^{2\qol+1}}
=\Omega\of{\frac{\eps^{4\qol+2}np^{2\qol+2}}{\log^{2\qol+1}n}} \gg \log n.
\]
So we may apply Lemmas \ref{firstOrder} and \ref{secondOrder} to get
\begin{align*}
 X & = \abs{\scr{B}} - s\pme{(\qol^2+z)}\frac{\abs{\scr{B}}}{\kappa^{\qol-1}}
 \pm s^2\frac{7q}{\kappa^{\qol}}\abs{\scr{B}} \\
  &= \abs{\scr{B}}\of{1- \frac{s\pme{(\qol^2+z+1)}}{\kappa^{\qol-1}} }
\end{align*}
where in the second line we use the fact that
$\frac{1}{\kappa} \ll \eps$.

Note that
\[
 \of{1-\frac{1}{\kappa^{\qol-1}}}^s =
 1 - \frac{s}{\kappa^{\qol-1}} + O\of{\frac{1}{\kappa^{2\qol-2}}}.
\]
Then by using $\frac{1}{\kappa} \ll \eps$ we get that
\begin{align*}
 X &= \pme{}\frac{n^d}{d!}p^s\of{\of{1-\frac{1}{\kappa^{\qol-1}}}^s
 \pm \frac{(2\qol+2)(\qol^2+\qol+2)}{\kappa^{\qol-1}}\cdot\eps } \\
   &= \frac{n^d}{d!}\of{p\of{1-\frac{1}{\kappa^{\qol-1}}}}^s\cdot\pme{}\of{1
   \pm\frac{(2\qol+2)(\qol^2+\qol +2)}{\kappa^{\qol-1}\of{1-\frac{1}{\kappa^{\qol-1}}}^s}\cdot\eps} \\
  &= \frac{n^d}{d!} \of{p\of{1-\frac{1}{\kappa^{\qol-1}}}}^s
  \pme{\of{1 + \frac{h(z)}{\kappa^{\qol-1}}}}
\end{align*}
where $h(\qol) = (2\qol+2)(\qol^2+\qol+3)$. Now $z\geq 2$
and so $h(\qol) \leq 7\qol^3$ which gives us the result
\[
 X=(1\pm\eps')\frac{n^d}{d!}p'^s
\]
with the desired probability.
\qed

\section{Finishing the proof of Theorem \ref{th1}}\label{finish}
 Let $H_0 = H$, $\eps_0 = \eps$ and $p_0=p$. Define $\eps_t$ and
 $p_t$ recursively using the following recursion:
\[
 \eps_{t+1} = \eps_t\of{1+7\qol^3\of{\frac{\eps_t^{2}}{6(k+1)\log n}}^{\qol-1}}
\]
and
\[
 p_{t+1} = p_t\of{1-\of{\frac{\eps_t^2}{6(k+1)\log n}}^{\qol-1}}.
\]

Let $T$ be the smallest index such that $p_T \leq \frac{1}{2}\eps^{\alpha}p$
where $\alpha = \frac{1}{9+7\qol^3}$.  For $t=0,\ldots,T$, let
$x_t = \of{\frac{\eps_t^2}{6(k+1)\log n}}^{\qol-1}$. Then since $(\eps_t)$ is an
increasing sequence, we have
\begin{align*}
\frac{1}{2}p\eps^{\alpha} \leq p_{T-1} &= \frac{p_{T-1}}{p_{T-2}}\cdot\frac{p_{T-2}}{p_{T-3}}\cdots
\frac{p_{2}}{p_{1}}\cdot\frac{p_1}{p}\cdot p \\
  &\leq p\of{1-x_0}^{T-1} \\
 & \leq pe^{-x_0(T-1)}.
\end{align*}
From this we can see that
\[
 T \leq O\of{\frac{\log^{\qol-1}n}{\eps^{2\qol-1}}} = o(n).
\]

Also note that since \[\of{1+7\qol^3x}(1-x)^{7\qol^3}
\leq e^{7\qol^3x}\of{e^{-x}}^{7\qol^3} =1,\] we have in general that
\[\frac{\eps_{t+1}}{\eps_t} = (1+7\qol^3x_t) \leq
\frac{1}{(1-x_t)^{7\qol^3}} = \of{\frac{p_t}{p_{t+1}}}^{7\qol^3}.\]
Hence
\begin{align*}
\eps_{T-1} &= \frac{\eps_{T-1}}{\eps_{T-2}}\cdot\frac{\eps_{T-2}}{\eps_{T-3}}
\cdots\frac{\eps_{2}}{\eps_{1}}\cdot\frac{\eps_{1}}{\eps}\cdot\eps \\
&\leq \eps\cdot\of{\frac{p_{T-2}}{p_{T-1}}\cdot\frac{p_{T-3}}{p_{T-2}}
\cdots\frac{p}{p_1}  }^{7\qol^3} \\
&=\eps\of{\frac{p}{p_{T-1}}}^{7\qol^3}\\
&<\eps\cdot\of{2\eps^{-\alpha}}^{7\qol^3} = \Theta\of{\eps^{1-{7\qol^3\alpha}}}
\end{align*}

So we have that
\[
 \eps_{T-1}^{1/8} = \Theta\of{ \eps^{\frac{9}{8}\alpha}} \ll \eps^\alpha.
\]

We now construct $H_1,\ldots,H_T$ such that each $H_t$ is $(\eps_t,p_t)$-regular.
Let $\kappa_t =\frac{6(k+1)\log n}{\eps_t^2}$ and $r=\frac{n^{k-2}q\ell }{k! p_t^{\qol-1}}\kappa_t$
and consider Procedure 1 applied to $H_t$ with these parameters. This produces
digraphs $D_{t,i}'$ and $k$-graphs $H_{t,i}'$ with all $H_{t,i}'$ disjoint. Let
$H_{t+1}$ be the $k$-graph which results from the deletion of all $H_{t,i}'$ from
$H_t$.  In order to apply Lemma \ref{everyDi'Uniform}
at each step, we must check that $\eps_t^{8\qol+2}np_t^{8\qol} \gg \log^{4\qol+1}n$.
This condition follows from our assumptions on $\eps,n,p$ since $\eps_t \geq \eps$ and
$p_t \geq \frac{1}{2}\eps^\alpha p$. So we have, with probability $1-o(n^{-1})$,
Procedure 2  results in the following properties:
\begin{itemize}
 \item Every $D_{t,i}'$ is $(12\qol^2\eps_t, \of{p_t/\kappa_t}^{\qol})$-regular.
 \item $H_{t+1}$ is $(\eps_{t+1},p_{t+1})$-regular.
\end{itemize}

Since $T=o(n)$, we may condition on this holding at each step. In order to apply the
result on packing cycles in digraphs to each $D_{t,i}'$, we must verify that
$\eps_t^{11} \n_q\of{p_t^\qol/\kappa_t^\qol}^8 \gg \log^5 n$. We have
\[
 \eps_t^{11} \n_q\of{\frac{p_t}{\kappa_t}}^{8z} \geq
 \Theta\of{\frac{\eps^{11 + 8\qol\alpha+16\qol}np^{8\qol}}{\log^{8\qol}n}} \gg \log^5n
\]
by our assumption that $\eps^{16z+12}np^{8z} \gg \log^{8z+5}n$ since $8\qol\alpha \leq 1$.
So every $D_{t,i}'$ can be packed with Hamilton cycles missing only
$(12\qol^2\eps_t)^{1/8}$-fraction of its edges.  As observed already,
these edge-disjoint Hamilton cycles in $D_{t,i}'$ correspond to edge 
disjoint Hamilton cycles in $H_{t,i}'$.
Hence the packing in $D_{t,i}'$ gives a packing in $H_{t,i}'$ missing the same fraction
of edges since there is a $\qol$-to-1 correspondence between edges in $H_{t,i}'$ and $D_{t,i}'$.

The above procedure is carried out until $H_T$ is created.  Then Hamilton cycles have
been packed in $H\backslash H_T$, up to an error of $(12\qol^2\eps_{T-1})^{1/8}$-fraction.
Let us estimate the fraction of edges present in $H_T$ itself. By applying
$(\eps,p)$-regularity to  $H$, we see that $H$ had at least
\[
 (1-\eps)\frac{n^k}{k!}p \geq \frac{n^k}{k!+1}p
\]
edges to begin with.

Similarly, we see that $H_T$ has at most
\[
 (1+\eps_T)\frac{n^k}{k!}p_T \leq (1+\eps_T)\frac{n^k}{2\cdot k!}\eps^\alpha p
 \leq \frac{n^k}{2\cdot k! -1}p \eps^\alpha
\]
edges. Since $k \geq 3$, we have that
\[
 \frac{\abs{H_T}}{\abs{H}} \leq c\eps^\alpha
\]
where $c < 1$ is some constant.

Hence the fraction of edges of $H$ not covered is at most
\[
 (12\qol^2\eps_{T-1})^{1/8}\cdot\of{1-c\eps^\alpha}+c\eps^\alpha
 \leq (12\qol^2\eps_{T-1})^{1/8} + c\eps^\alpha \leq  \eps^\alpha,
\]
since $\eps_{T-1}^{1/8} \ll \eps^\alpha$.

\end{document}